\RequirePackage{etex}
\RequirePackage{easybmat}

\documentclass[graybox]{svmult}

\usepackage{mathptmx}       
\usepackage{helvet}         
\usepackage{courier}        
\usepackage{type1cm}        
\usepackage{makeidx}         
\usepackage{graphicx}        
\usepackage{multicol}        
\usepackage[bottom]{footmisc}

\makeindex             

\newcommand{\DS}[0]{{\it{Z}}}

\newcommand{\NN}{{[n]}}

\newcommand{\cPsi}{\Omega}

\newcommand{\indic}{\Phi}
\newcommand{\clip}{{\mbox{clip}}}

\newcommand{\K}[0]{{\it C}}

\newcommand{\Kidx}[0]{{\it c}}

\newcommand{\vsubset}[2]{#1_{[#2]}}

\newcommand{\dom}{\textrm{dom}}

\newcommand{\mysgn}{\textrm{sgn}}

\newcommand{\successor}{\textrm{succ}}
\newcommand{\predecessor}{\textrm{pred}}

\newcommand{\TT}[0]{{\tau}}

\newcommand{\R}{\mathbf{R}}

\newcommand{\Prob}{\mathbf{Prob}}

\newcommand{\eqdef}{:=}

\newcommand{\vc}[2]{#1^{(#2)}}
\newcommand{\bvc}[2]{#1^{(#2)}}
\newcommand{\vt}[2]{#1_{#2}}

\newcommand{\mrow}[2]{#1_{{#2}{:}}}
\newcommand{\mcol}[2]{#1_{{:}{#2}}}
\newcommand{\mel}[3]{#1_{{#2},{#3}}}

\newcommand{\Zj}[2]{\vc{Z_{#2}}{#1}}

\newcommand{\ncs}[2]{\|#1\|^2_{(#2)}}

\newcommand{\calJ}{\mathcal{J}}

\newcommand{\RT}{\mathcal{T}}

\newcommand{\BO}[1]{\mathcal {O}\left({#1}\right)}

\newcommand{\nbp}[2]{\|#1\|_{(#2)}}   
\newcommand{\nbd}[2]{\|#1\|_{(#2)}^*} 

\newcommand{\lf}{\mathcal L}
\renewcommand{\E}{\mathbf{E}}
\newcommand{\U}{U}
\newcommand{\N}{N}

\newcommand{\Partc}{\it{P}}

\newcommand{\setn}{[n]}
\newcommand{\Lip}{L}

\newcommand{\ve}[2]{\langle #1 ,  #2 \rangle}

 \newtheorem{assumption}[theorem]{Assumption}

\let\la=\langle
\let\ra=\rangle

\usepackage[ruled,vlined,linesnumbered]{algorithm2e}

\SetAlgorithmName{Algorithm Schema}{\autoref}{List of Algorithms}

\SetKwFor{ParallelForEach}{for each}{in parallel do}{endfor}
\SetKwFor{ParallelCForEach}{for each computer}{in parallel do}{endfor}
\SetKwFor{ParallelTForEach}{for each thread}{in parallel do}{endfor}
\SetKwFor{WhileNotSatisfied}{while termination criteria are not satisfied}{}{}%

\usepackage{amssymb}

\usepackage{amsmath}

\begin{document}

\title*{Distributed Block Coordinate Descent  for Minimizing Partially Separable Functions}


\titlerunning{Distributed Block Coordinate Descent  for Minimizing Partially Separable Functions}

\author{Jakub Mare\v{c}ek, Peter Richt\'arik and Martin Tak\'a\v{c}}

\institute{ Jakub Mare\v{c}ek
\at IBM Research -- Ireland, Dublin, Ireland, \email{jakub.marecek@ie.ibm.com}
\and Peter Richt\'arik
\at School of Mathematics, University of Edinburgh,  UK, \email{peter.richtarik@ed.ac.uk}
\and
Martin Tak\'a\v{c} 
\at Dept.\ of Industrial \& Systems Engineering, Lehigh University,  USA, \email{takac.mt@gmail.com}\\[3mm]
The first author was supported by EPSRC grant EP/I017127/1 (Mathematics for Vast Digital Resources) in 2012 and
by the EU FP7 INSIGHT project (318225) subsequently.
The second author was supported by EPSRC grant EP/I017127/1.
The third author was supported by the Centre for Numerical Algorithms and Intelligent Software, funded by EPSRC grant EP/G036136/1 and the Scottish
Funding Council.
}


%
%
\maketitle

\abstract{
A distributed randomized block coordinate descent method for minimizing a convex function of a huge number of variables is proposed. 
The complexity of the method is analyzed under the assumption that the smooth part of the objective function is partially block separable.
The number of iterations required 
 is bounded by a function of the error and the degree of separability,
 which extends the results in \cite{RT:PCDM} to a distributed environment. 
Several approaches to the distribution and synchronization of the computation across a cluster of multi-core computer are described and promising computational results are provided.
}

\section{Introduction}

With the ever increasing availability of data comes the need to solve ever larger instances of problems in data science and  machine learning, many of which turn out to be convex optimization problems of enormous dimensions. A single  machine is often unable to store the complete data in its main memory. This suggests the need for efficient algorithms, which can benefit from distributing the data and computations across many computers. 

In this paper, we study optimization problems of the form:
\begin{equation}\label{eq:P}
\min_{x\in \R^N} \left[ F(x) \eqdef f(x) + \cPsi(x)\right],
\end{equation}
where $f$ is a smooth, convex and partially block separable function, and $\cPsi$ is  a possibly non-smooth, convex, block separable, and ``simple'' extended real valued function. The technical definitions of these terms are given in Section~\ref{sec:block_structure}.

\subsection{Contributions}

We propose and study the performance of a {\em distributed block coordinate descent method} applied to problem \eqref{eq:P}.

In our method, the blocks of coordinates are first partitioned among  $\K$ computers of a cluster. Likewise, data associated with these blocks are partitioned accordingly and stored in a distributed way.
 In each of the subsequent iterations, each computer chooses $\TT$ blocks out of those stored locally, uniformly at random. Then, each computer computes and applies an update to the selected blocks, in parallel, out of information available to it locally.  An update, which happens to be the residual in data-fitting problems, is then transmitted to other computers, which receive it either by the beginning of the next iteration or at some later time.
In the former case, we denote the methods ``sychronous'' and we analyse them in detail. In the latter case, we denote the methods ``asynchronous'' and we include them for the sake of comparison in Section~\ref{para:distributed-computation}.

The main contributions of this paper are, in no particular order:
\begin{enumerate}
\item \textbf{Partial separability.}  This is the first time such a distributed block-coordinate descent  method is analyzed under the assumption that $f$ is partially separable.


\item \textbf{New step-length.} Our method and analysis is based on an  expected separable overapproximation (ESO) inequality for partially separable functions and distributed samplings in Theorem~\ref{thm:esoGeneral} in Section~\ref{sec:ESO}. The length of the step we take in each iteration is given by the optimum of this ESO.

\item {\bf Iteration complexity.} We show that the iteration complexity of the method depends on the degree of block separability of $f$: the more separable the instance, the fewer iterations the method requires. The complexity results are stated in two theorems in Section~\ref{sec:convergence} and are of the order of $O(\log (1/\epsilon))$ for strongly 
convex $F$ and $O(1/\epsilon)$ for general convex $F$.
At the same time, the separability also reduces the run-time per iteration.

\item {\bf Efficient implementation.} When we replace the natural synchronous communications between computers, as analysed in Section~\ref{sec:convergence}, with asynchronous communication, we obtain a major speed-up in the computational performance. An efficient open-source implementation of both synchronous and asynchronous methods is available as part of the package \url{http://code.google.com/p/ac-dc/}.

\end{enumerate}

Our method and results are valid not only for a cluster setting, where there really are $\K$ computers which do not share any memory,
and hence have to communicate by sending messages to each other, but also for computers using the Non-Uniform Memory Access (NUMA) architecture,
where the memory-access time depends on the memory location relative to a processor,
and accessing local memory is much faster than accessing memory elsewhere.
NUMA architectures are increasingly more common in multi-processor machines.


\subsection{Related work}

Before we proceed, we give a brief overview of some existing literature on coordinate descent methods.
For further references, we refer the reader to \cite{Bertsekas1989,RT:PCDM,fercoq2013accelerated}. 

{\bf Block-coordinate descent.}
  Block-coordinate descent is a simple iterative optimization strategy, where
two subsequent iterates differ only in a single block of coordinates. In a very common special case, 
each block consists of a single coordinate. The choice of the block can be deterministic, e.g., cyclic (\cite{Saha10finite}), greedy (\cite{RT:TTD2011}), or randomized. 
Recent theoretical guarantees for randomized coordinate-descent algorithms can be found in \cite{Nesterov:2010RCDM,RT:UCDC,fercoq2013smooth,lu2013complexity,necoara2013distributed,lee2013efficient}.
Coordinate descent algorithms are also closely related to 
coordinate relaxation, linear and non-linear Gauss-Seidel methods, subspace correction, and domain decomposition (see \cite{Bertsekas-Book} for references). For  classical references on non-randomized variants, we refer to the work of Tseng \cite{Tseng:CCMCDM:Smooth, Tseng:CGDM:Nonsmooth, Tseng:CGDMLC:Nonsmooth, Tseng:CBCDM:Nonsmooth}.


{\bf Parallel block-coordinate descent.}
Clearly, one can parallelize coordinate descent by updating several blocks in parallel. The related complexity issues were studied by a number of authors. Richt\'{a}rik and Tak\'{a}\v{c} studied a broad class of parallel methods for the same problem we study in this paper, and introduced the concept of ESO \cite{RT:PCDM}. The complexity was improved by Tappenden et al. \cite{rachael:Improved}.
An efficient accelerated version was introduced by Fercoq and Richt\'{a}rik \cite{fercoq2013accelerated} and an inexact version was studied in \cite{T2}. An asynchronous variant was studied by Liu et al. \cite{liu2013asynchronous}.
A non-uniform sampling and a method for dealing with non-smooth functions were described in \cite{richtarik2013optimal} and \cite{fercoq2013smooth}, respectively. Further related work can be found in \cite{scherrer2012feature,T1,zhao2014stochastic,hogwild}.

{\bf Distributed block-coordinate descent.}
Distributed coordinate descent was first proposed
by Bertsekas and Tsitsiklis \cite{Bertsekas1989}.
The literature on this topic was rather sparse, c.f. \cite{DCD}, 
until the research presented in this paper raised the interest, 
which lead to the analyses of Richt\'{a}rik and Tak\'{a}\v{c} \cite{richtarik2013distributed} and 
Fercoq et al. \cite{fercoq2014fast}.
These papers do not consider blocks, and specialise our results to  convex functions admitting a quadratic upper bound. 

In the machine-learning community, distributed algorithms have been studied for particular problems, e.g., training of support vector machines \cite{6122723}.
Google \cite{Chang_psvm:parallelizing} developed a library called PSVM, where parallel  row-based incomplete Cholesky factorization is employed
in an interior-point method. A MapReduce-based distributed algorithm for SVM was found to be effective in automatic image annotation \cite{Alham20112801}.
Nevertheless, none of these papers use coordinate descent.

\section{Notation and assumptions}\label{sec:block_structure}

In this section, we introduce the notation used in the rest of the paper and state our assumptions formally. 
We aim to keep our notation consistent with that of Nesterov  \cite{Nesterov:2010RCDM} and Richt\'{a}rik \& Tak\'{a}\v{c} \cite{RT:PCDM}.

{\bf Block structure.} We  decompose  $\R^\N$ into $n$ subspaces as follows. Let $\U\in \R^{\N\times \N}$ be the $\N\times \N$ identity matrix and further let $\U = [\U_1,\U_2,\dots,\U_n]$ be a column decomposition of $\U$ into $n$ submatrices, with $\U_i$ being of size $\N \times \N_i$, where $\sum_i \N_i = \N$.
It is easy to observe that any vector $x\in \R^\N$ can be written uniquely as $x = \sum_{i=1}^n \U_i x^{(i)}$, where $x^{(i)} \in \R^{\N_i}$.
Moreover, $x^{(i)}=\U_i^T x$.
In view of the above, from now on we  write $x^{(i)}\eqdef U_i^T x \in \R^{\N_i}$, and call $x^{(i)}$ the \emph{block} $i$ of $x$.  

{\bf Projection onto a set of blocks.} Let us denote 
$\{1, 2, \ldots, n\}$ by $[n]$, a set of blocks $S\subseteq \NN$, $x\in \R^\N$, and let $x_{[S]}$ be the vector in $\R^\N$ whose blocks $i\in S$ are identical to those of $x$, but whose other blocks are zeroed out. Block-by-block, we thus have
$(x_{[S]})^{(i)} = x^{(i)}$ for $ i\in S$ and  $(x_{[S]})^{(i)} =0 \in\R^{\N_i}$, otherwise. It will be more useful to us however to write
\begin{equation}\label{eq:lllop09}x_{[S]} \eqdef \sum_{i\in S} U_i x^{(i)},\end{equation}
where we adopt the convention that if $S=\emptyset$, the sum is equal  $0\in \R^\N$.

{\bf Norms.} Spaces $\R^{\N_i}$, $i \in \NN$, are equipped with a pair of conjugate norms: $\|t\|_{(i)}$ and $\|t\|_{(i)}^* \eqdef \max_{\|s\|_{(i)}\leq 1} \ve{s}{t}$, $t\in \R^{\N_i}$. For $w\in \R^n_{>0}$, where $\R_{>0}$ is a set of positive real numbers, define a pair of conjugate norms in $\R^\N$ by
\begin{equation}\label{eq:norms} \|x\|_w = \left[\sum_{i=1}^n w_i \ncs{\vc{x}{i}}{i}\right]^{1/2}, \quad
\|y\|_w^* \eqdef \max_{\|x\|_w\leq 1} \ve{y}{x} = \left[\sum_{i=1}^n w_i^{-1} ( \nbd{y^{(i)}}{i})^2\right]^{1/2}.\end{equation}

We shall assume throughout the paper that $f$ has the following properties.

\begin{assumption}[Properties of $f$] \label{ass:f}
Function $f: \R^\N\to \R$ satisfies:
\begin{enumerate}
\item \textbf{Partial separability.} Function $f$ is of the form 
\begin{equation}\label{eq:iuhs8sis}
f(x) = \sum_{J\in \calJ} f_J (x),
\end{equation}
where  $\calJ$ is a collection of subsets of $\setn$ and function $f_J$ depends on $x$ through blocks $x^{(i)}$ for $i \in J$ only. The quantity
$\omega \eqdef \max_{J \in \calJ} |J|$ is the  degree
of separability of $f$.
\item \textbf{Convexity.} Functions $f_J$, $J \in \calJ$ in \eqref{eq:iuhs8sis} are convex.
\item \textbf{Smoothness.} The gradient of $f$ is block  Lipschitz, uniformly in $x$, with positive constants $\Lip_1,\dots,\Lip_n$. That is,  for all $x\in \R^\N$, $i\in \NN$ and $t\in \R^{\N_i}$,
\begin{equation}\label{eq:f_iLipschitzder}
 \nbd{\nabla_i f(x+\U_i t)-\nabla_i f(x)}{i} \leq \Lip_i \nbp{t}{i},
  \end{equation}
where $\nabla_i f(x)  \eqdef (\nabla f(x))^{(i)} = \U^T_i \nabla f(x) \in \R^{\N_i}$.

\end{enumerate}
\end{assumption}

A few remarks are in order:
\begin{enumerate}
\item Note that every function $f$ is trivially of the form 
\eqref{eq:iuhs8sis}: we can always assume that $\calJ$ contains just the single set $J=\setn$ and let $f_J = f$. In this case we would have $\omega=n$. However, many functions appearing in applications can naturally be decomposed as a sum of a number of functions each of which depends on a small number of blocks of $x$ only. That is, many functions have degree of separability  $\omega$ that is much smaller than $n$. 


\item Note that since $f_J$ are convex, so is $f$. While it is possible to remove this assumption and provide an analysis in the non-convex case, this is beyond the scope of this paper.

\item
An important consequence of \eqref{eq:f_iLipschitzder} is the following standard inequality \cite{NesterovBook}:
\begin{equation}\label{eq:Lipschitz_ineq}f(x+\U_i t) \leq f(x) + \ve{\nabla_i f(x)}{t} + \tfrac{\Lip_i}{2}\nbp{t}{i}^2.\end{equation}
\end{enumerate}

\begin{assumption}[Properties of $\cPsi$] \label{ass:reularizer}
We assume that $\cPsi: \R^\N\to \R\cup \{+\infty\}$ is (block) separable, i.e., that it can be decomposed as follows:
\begin{equation}\label{eq:Psi_block_def}
  \cPsi(x)=\sum_{i=1}^n \cPsi_i(x^{(i)}),
\end{equation}
where the functions $\cPsi_i:\R^{\N_i}\to \R \cup \{+\infty\}$ are convex and closed.
\end{assumption}


\section{Distributed block coordinate descent method}
\label{sec:algorithm}

In this section we describe our distributed block coordinate descent method (Algorithm~\ref{alg:distributedEx}). It is designed to solve convex optimization problems of the form \eqref{eq:P}, where the data describing the instance are so large that it is impossible to store these in  memory of a single computer. 

\begin{algorithm}[th!]
\caption{Distributed Block Coordinate Descent} 
\label{alg:distributedEx}
\DontPrintSemicolon
choose $\vt{x}{0} \in \R^N$\;
$k \longleftarrow 0$\;
\WhileNotSatisfied{}{
  $\vt{x}{k+1} \longleftarrow \vt{x}{k}$\;
  \ParallelCForEach{ $\Kidx \in \{ 1, \ldots, \K \} $}{
    sample a set of coordinates $\Zj{\Kidx}{k} \subseteq \vc{\Partc}{\Kidx}$ of size $\tau$, uniformly at random 

    \ParallelTForEach{ $i \in \Zj{\Kidx}{k}$}{
      compute an update $\bvc{h}{i}(\vt{x}{k})$\;
      $\vt{x}{k+1} \longleftarrow \vt{x}{k+1} + U_i \bvc{h}{i}(\vt{x}{k})$ \; 
    }
  }
  $k \longleftarrow k + 1$
}
\end{algorithm}



\textbf{Pre-processing.} Before the method is run,  the set of blocks is partitioned into $\K$ sets $P^{(c)}$, $c=1,2,\dots,C$. Each computer ``owns'' one partition and will only store and update blocks of $x$ it owns. That is, the blocks $i\in P^{(c)}$ of $x$ are stored on and updated by computer $c$ only.  Likewise,  ``all data'' relevant to these blocks are stored on computer $c$. 
We deal with the issues of data distribution and communication only in Section~\ref{sec:implementation}.

\textbf{Distributed sampling of blocks.} In Step 6 of Algorithm ~\ref{alg:distributedEx}, each computer $\Kidx$ chooses  a random subset $Z_k^{(c)}$ of blocks from its partition $\vc{\Partc}{\Kidx}$.  We assume that $|Z_k^{(c)}|=\tau$, and that it  is chosen uniformly at random from all subsets of $P^{(c)}$ of cardinality $\tau$. Moreover, we assume the choice is done independently from all history and from what the other computers do in the same iteration. Formally, we say that the set of blocks chosen by all computers in iteration $k$, i.e., 
$Z_k = \cup_{c=1}^C Z_k^{(c)}$,
 is a $(C,\tau)$-{\em distributed sampling.} 
 
For easier reference in the rest of the paper, we formalize the setup described above as Assumption~\ref{asm:partitions} at the end of this section (where we drop the subscript $k$, since the samplings are independent of $k$).

\textbf{Computing and applying block updates.} In Steps 7-9, each computer $c$ first computes and then applies updates to blocks $i \in Z_k^{(c)}$ to $x_k$. This is done on each computer in parallel. Hence, we have two levels of parallelism: across the nodes/computers and within each computer. The update to block $i$ is denoted by $\bvc{h}{i}(\vt{x}{k})$ and arises as a solution of an optimization problem in the lower dimensional space $R^{N_i}$:
 \begin{equation}\label{eq:h_definitionSeparablePart}
h^{(i)}(x_k) \leftarrow \arg\min_{t \in \R^{\N_i}}
\la \nabla_i f(x_k), t\ra  +\frac{\beta w_i}{2} \ncs{t}{i} +\cPsi_i(\vc{x_k}{i}+t).
 \end{equation}
Our method is most effective when this optimization problem has a closed form solution, which is the case in many applications. Note that \emph{nearly all} information that describes problem \eqref{eq:h_definitionSeparablePart} for $i \in P^{(c)}$ is available at node $c$. In particular, $x_k^{(i)}$ is stored on $c$. Moreover, we can store the description of $\Omega_i$, norm $\|\cdot\|_{(i)}$ and the pair $(\beta,w_i)$, for $i \in P^{(c)}$,  on node $c$ and only there. 

Note that we did not specify yet the values of the parameters $\beta$ and $w = (w_1,\dots,w_n)$. These depend on the properties of $f$ and sampling $\hat{Z}$. We shall give theoretically justified formulas for these parameters in Section~\ref{sec:ESO}.

\textbf{Communication.} Finally, note that in order to find $h^{(i)}(x_k)$, each computer needs to be able to compute $\nabla_i f(x_k)$ for blocks $i \in Z_k^{(c)}\subseteq P^{(c)}$. This is the only information that an individual computer can \emph{not} obtain from the data stored locally. We shall describe an efficient communication protocol that allows each node to compute $\nabla_i f(x_k)$ in Section~\ref{sec:implementation}.

 \begin{assumption}[Distributed sampling]\label{asm:partitions} We make the following assumptions:
\begin{enumerate}
\item \textbf{Balanced partitioning.} The set of blocks is partitioned into $\K$  groups $\vc{\Partc}{1},\dots, \vc{\Partc}{\K}$, each of size $s\eqdef n/\K$. That is,
\begin{enumerate}
\item $\{1,2,\dots,n\} = \cup_{\Kidx=1}^\K \vc{\Partc}{\Kidx}$,
\item $\vc{\Partc}{c'}\cap \vc{\Partc}{c''} = \emptyset$ for $c'\neq c''$,
\item $|\vc{\Partc}{c}| =:s$ for all $c$.
\end{enumerate}
\item \textbf{Sampling.} For each $c \in\{1,\dots,C\}$, the set $\hat{Z}^{(c)}$ is a random subset of $P^{(c)}$ of size $\tau \in \{1,2,\dots,s\}$, where each subset of size $\tau$ is chosen with equal probability.
\end{enumerate}

We refer call the random set-valued mapping $\hat{\DS} \eqdef \cup_{c=1}^C \hat{\DS}^{(c)}$ by the name $(C,\tau)$-distributed sampling.
\end{assumption}

\section{Expected separable overapproximation (ESO)}
\label{sec:ESO}

The following concept was first  defined in \cite{RT:PCDM}. It plays a key role in the complexity analysis of randomized coordinate descent methods.

\begin{definition}[ESO] Let $\hat \DS$ be any uniform sampling, i.e., a random sampling of blocks for which $\Prob(i \in \hat{\DS}) = \Prob(j \in \hat{\DS})$ for all $i,j \in \setn$. We say that function $f$ admits an ESO 
 with respect to sampling $\hat\DS$, with parameters $\beta>0$ and $w\in\R^n_{>0}$, if the following inequality holds for all $x, h \in \R^\N$:
\begin{equation}\label{eq:ss8587sj}
  \E[f(x+\vsubset{h}{\hat \DS})]
  \leq f(x) + \tfrac{\E[|\hat \DS|]}{n} \left(\la \nabla f(x), h\ra + \tfrac{\beta}{2} \|h\|_w^2    \right).
\end{equation}
For simplicity, we will sometimes write  $(f,\hat{\DS})\sim ESO(\beta,w)$.
\end{definition}

In the rest of this section we derive an ESO inequality for $f$ satisfying Assumption~\ref{ass:f} (smooth, convex, partially separable) and for sampling $\hat{Z}$ satisfying Assumption~\ref{asm:partitions} ($(C,\tau)$-distributed sampling). This has not been done before in the literature. In particular, we give simple closed-form formulas for parameters $\beta$ and $w$, which we shall use in Section~\ref{sec:convergence} to shed light on the performance of the method.

We first need to establish an auxiliary result. We use $[n]$ to denote 
$\{ 1, 2, \ldots, n \}$.

\begin{lemma}Let $\hat{\DS} = \cup_{c=1}^\K \hat{\DS}^{(c)}$ be a $(\K,\tau)$-distributed sampling. Pick $J \subseteq [n]$ and assume that $|\vc{\Partc}{\Kidx} \cap J |= \xi$ for some $\xi\geq 1$ and all $\Kidx$. Let $\kappa = \kappa(|\hat{\DS}\cap J|,i)$ be any function that depends on $|\hat{\DS}\cap J|$ and $i \in \setn$ only.
Then
\begin{equation}
\label{eq:jd9876gdh}\E \left[ \sum_{i \in \hat{\DS}\cap J} \kappa(|\hat{\DS} \cap J|,i)\right] = \E\left[ \frac{|\hat{\DS}\cap J|}{\K \xi} \sum_{i \in J} \kappa(|\hat{\DS} \cap J|,i)\right]. \end{equation}
\end{lemma}

\begin{proof} 
Let us denote by $\vc{J}{\Kidx} = J\cap \vc{\Partc}{\Kidx}$, $\zeta = |\hat \DS \cap J|$ and $\vc{\zeta}{\Kidx} = |\hat \DS \cap \vc{J}{\Kidx}|$. Then
\begin{eqnarray*}
\E \left[ \sum_{i \in \hat{\DS}\cap J} \kappa(\zeta,i)\right]
&=&
\E\left[\E\left[ \sum_{i \in \hat{\DS}\cap J} \kappa(\zeta,i)
 \;|\;    \zeta
\right]
\right]
\\
&=&
\E\left[\E\left[
\E \left[ \sum_{i \in \hat{\DS}\cap J} \kappa(\sum_{\Kidx = 1}^\K \vc{\zeta}{\Kidx},i)
 \;|\;  \vc{\zeta}{1},\dots,\vc{\zeta}{\K},  \sum_{\Kidx = 1}^\K \vc{\zeta}{\Kidx} = \zeta
\right]
 \;|\;  \zeta
\right]
\right]
\\
&=&
\E\left[\E\left[
\E \left[ \sum_{\Kidx=1}^\K
  \sum_{i \in \hat{\DS}^{(\Kidx)}\cap \vc{J}{\Kidx}} \kappa(\zeta,i)
 \;|\;  \vc{\zeta}{1},\dots,\vc{\zeta}{\K}
\right]
 \;|\;  \sum_{\Kidx = 1}^\K \vc{\zeta}{\Kidx} = \zeta
\right]
\right]
\\
&=&
\E\left[\E\left[
\sum_{\Kidx=1}^\K
 \frac{\vc{\zeta}{\Kidx}}{\xi} 
  \sum_{i \in  \cap \vc{J}{\Kidx}} \kappa(\zeta,i)
   \;|\;  \sum_{\Kidx = 1}^\K \vc{\zeta}{\Kidx} = \zeta
\right]
\right]
\\
&=&
\E\left[
\sum_{\Kidx=1}^\K
 \frac{ \zeta }{\xi \K} 
  \sum_{i \in  \cap \vc{J}{\Kidx}} \kappa(\zeta,i)
\right]
=
\E\left[
 \frac{ \zeta }{\xi \K} 
  \sum_{i \in    J} \kappa(\zeta,i)
\right].\qquad \qquad\qed
\end{eqnarray*}

\end{proof}


The main technical result of this paper follows. This is a generalization of a result from \cite{RT:PCDM} for partially separable $f$ and $\tau$-nice sampling to the distributed ($c>1$) case. Notice that for $\K=1$ we have  $\xi=\omega$.

\begin{theorem}[ESO]
\label{thm:esoGeneral}
Let $f$ satisfy Assumption~\ref{ass:f} and 
$\hat\DS$ satisfy Assumption \ref{asm:partitions}. Let \footnote{Note that $\xi \in  \{
\lceil\tfrac\omega\K\rceil,\dots,\omega\}$.
} $\xi\eqdef \max \{|\vc{\Partc}{\Kidx}\cap J| \;:\; c\in \{1,\dots,\K\}, \;J\in \calJ \}$.
Then $(f, \hat \DS)$ admits ESO with parameters $\beta$ and $w$ given by 
\begin{equation}
 \label{eq:betaForGeneralCase}
  \beta =1 +\frac{(\xi-1)(\TT-1)}{\max\{1,s-1\}}
 +    (\K-1)
   \frac{\xi \TT}{s },
\end{equation}
and $w_i = L_i$, $i=1,2,\dots,n$.
\end{theorem}

\begin{proof}
For fixed $x \in \R^\N$,  define $\phi(h) \eqdef f(x+h)-f(x) - \la \nabla f(x),h \ra$.
Likewise, for all $J \in \calJ$ we define $\phi_J(h)\eqdef f_J(x+h) - f_J(x) - \ve{\nabla f_J(x)}{h}$. Note that
\begin{equation}\label{eq:sjs65876}
\phi(h) = \sum_{J \in \calJ} \phi_J(h).
\end{equation}
Also note that the functions $\phi_J$ and $\phi$ are convex and minimized at $h=0$, where they attain the value of $0$. For any uniform sampling, and hence for $\hat{\DS}$ in particular, and any $a\in \R^\N$, one has
$
\E[\ve{a}{h_{[\hat{\DS}]}}] = \tfrac{\E[|\hat{\DS}|]}{n}\ve{a}{h},
$
and therefore
\begin{equation}\label{e:asfsafda}
\E[\phi(h_{[\hat{\DS}]})] = \E[f(x+h_{[\hat{\DS}]})] - f(x) - \tfrac{\E[|\hat{\DS}|]}{n} \ve{\nabla f(x)}{h}.
\end{equation}
Because of this, and in view of \eqref{eq:ss8587sj} and the fact  that as $\E[|\hat{\DS}|] = \K \tau$,\footnote{In fact, $|\hat{\DS}|=\K \tau$ with probability 1.} we only need to show that
\begin{equation}
\label{eq:jsus868s}
\E[\phi(h_{[\hat{\DS}]})]	 \leq \tfrac{\K \tau}{n} \tfrac{\beta}{2}\|h\|_w^2.
\end{equation}
Our starting point in establishing \eqref{eq:jsus868s} will be the observation that
from \eqref{eq:Lipschitz_ineq} used with $t=h^{(i)}$
we get
\begin{equation}\label{eg:asfoipowpf2}
 \phi(U_i h^{(i)}) \leq \tfrac{\Lip_i}{2} \ncs{\vc{h}{i}}{i}, \quad i\in [n].
\end{equation}

To simplify the proof, we shall without loss of generality  assume that $|\vc{\Partc}{\Kidx}\cap J| = \xi$ for all $c\in \{1,2,\dots,\K\}$ and $J \in \calJ$ for some constant $\xi>1$. This can be achieved by extending the sets  $J \in \calJ$ by introducing dummy dependencies (note that the assumptions of the theorem are still satisfied after this change). For brevity, let us write $\theta_{J,\hat \DS} \eqdef |J\cap \hat \DS|$ and $h_{[i]}\eqdef U_i h^{(i)}$. Fixing  $J \in \calJ$ and $h \in \R^\N$, we can estimate:

\begin{eqnarray}
\E[\phi_J(\vsubset{h}{\hat \DS})]
 &\overset{\eqref{eq:lllop09}}{=}& \E\left[ 
   \phi_J\left(\sum_{i \in \hat \DS }  \vsubset{h}{i}\right)\right]
 \;=\; \E\left[ 
   \phi_J\left(\sum_{i \in \hat \DS \cap J}  \vsubset{h}{i}\right)\right]
   \notag
\\
 &=& \E\left[ 
   \phi_J\left(\tfrac1{\theta_{J,\hat \DS}}\sum_{i \in \hat \DS \cap J} \theta_{J,\hat \DS} \vsubset{h}{i}\right)\right]
   \; \leq \;
 \E\left[ 
   \tfrac1{\theta_{J,\hat \DS}}\sum_{i \in \hat \DS \cap J} \phi_J\left( \theta_{J,\hat \DS} \vsubset{h}{i}\right)\right]
   \notag
\\
&\overset{\eqref{eq:jd9876gdh}}{=}&
 \E\left[
 \tfrac1{\theta_{J,\hat \DS}}
   \left(\tfrac{\theta_{J,\hat \DS}}{\K \xi}
   \sum_{i \in   J} \phi_J\left( \theta_{J,\hat \DS} \vsubset{h}{i}\right)\right)\right]
   \;=\;
 \frac{1}{\K \xi}
 \E\left[
   \sum_{i \in   J} \phi_J\left( \theta_{J,\hat \DS} \vsubset{h}{i}\right) \right]\notag\\
& =&
 \frac{1}{\K \xi}
 \E\left[
   \sum_{i \in  \setn} \phi_J\left( \theta_{J,\hat \DS} \vsubset{h}{i}\right) \right].\label{eq:sgs78jjs8s}
\end{eqnarray}

In the second equation above we have used the assumption that $\phi_J$ depends on blocks $i \in J$ only. The only inequality above follows from convexity of $\phi_J$. Note that this step can only be performed if the sum is over a nonempty index set, which happens precisely when $\theta_{J,\hat{\DS}} \geq 1$. This technicality can be handled at the expense of introducing a heavier notation (which we shall not do here), and  \eqref{eq:sgs78jjs8s} still holds. Finally, in one of the last steps we have used \eqref{eq:jd9876gdh} with $\kappa(|\hat{\DS}\cap J|,i) \leftarrow \phi_J(\theta_{J,\hat{\DS}}h_{[i]})$.

By summing up inequalities \eqref{eq:sgs78jjs8s} for $J \in \calJ$, we get
\begin{eqnarray}
\E\left[ \phi(\vsubset{h}{\hat \DS})\right]
&\overset{\eqref{eq:sjs65876}}{=}& 
\sum_{J \in \calJ} \E\left[ \phi_J(\vsubset{h}{\hat \DS})\right]
\; \overset{\eqref{eq:sgs78jjs8s}}{\leq} 
\;
 \frac{1}{\K \xi}\sum_{J \in \calJ}
 \E\left[
   \sum_{i \in \setn} \phi_J \left( \theta_{J,\hat \DS} \vsubset{h}{i}\right)\right]
\notag
\\ 
&\overset{\eqref{eq:sjs65876}}{=}&  
 \frac{1}{\K \xi}
 \E\left[
   \sum_{i \in \setn} \phi \left( \theta_{J,\hat \DS} \vsubset{h}{i}\right)\right]\;\overset{\eqref{eg:asfoipowpf2}}{\leq}\;
 \frac{1}{\K \xi}
 \E\left[
   \sum_{i \in  \setn}
   \tfrac{\Lip_i}{2}
   \ncs{\theta_{J,\hat \DS} h^{(i)}}{i}  \right]\notag
\\
&=& 
\frac{1}{2\K \xi}
 \E\left[
   \theta_{J,\hat \DS}^2 \sum_{i \in  \setn}
   \Lip_i
   \ncs{h^{(i)}}{i} \right]
\;\; \overset{\eqref{eq:norms}}{=}\;\;
\frac{1}{2 \K \xi} \|h\|_w^2
\E\left[\theta_{J,\hat \DS}^2\right].\label{eq:sjsjs987}
\end{eqnarray}

We now need to  compute $\E[\theta_{J,\hat \DS}^2]$.
Note that the random variable $\theta_{J,\hat \DS}$
is the sum of $\K$ independent random variables
$\theta_{J,\hat \DS} = \sum_{\Kidx=1}^\K \theta_{J,\vc{\hat \DS}{\Kidx}}$,
where  $\theta_{J,\vc{\hat \DS}{\Kidx}}$ has the simple law
$$\Prob(\theta_{J,\vc{\hat \DS}{\Kidx}}=k)=
\begin{pmatrix}\xi \\ k \end{pmatrix}\begin{pmatrix}s-\xi \\ \TT-k \end{pmatrix}/
\begin{pmatrix}s\\ \TT \end{pmatrix}.
$$
We therefore get
\begin{eqnarray}
\E[\theta_{J,\hat \DS}^2]
&=& \E[(\sum_{\Kidx=1}^\K \theta_{J,\vc{\hat \DS}{\Kidx}})^2]
\; = \; \K \E[(\theta_{J,\vc{\hat \DS}{\Kidx}})^2]
 +   \K(\K-1)
  (\E[\theta_{J,\vc{\hat \DS}{\Kidx}}])^2 \notag
\\
&=& \K \tfrac{\xi \TT}{s} \left(1 +\tfrac{(\xi-1)(\TT-1)}{\max\{1,s-1\}}  \right)
 +   \K(\K-1)
  \left( \tfrac{\xi}{s} \TT  \right)^2.\label{eq:js45sgs78ts}
\end{eqnarray}

It only remains to combine  \eqref{eq:sjsjs987}   and \eqref{eq:js45sgs78ts} to get \eqref{eq:jsus868s}. \qquad \qquad \qquad \qquad \qquad \qed
\end{proof}

Note that ESO inequalities have recently been used in the analysis of distributed coordinate descent methods by Richt\'{a}rik and Tak\'{a}\v{c} \cite{richtarik2013distributed} and Fercoq et al. \cite{fercoq2014fast} However, their assumptions on $f$ and derivation of ESO are very different and hence our results apply to a different class of functions.

\section{Iteration complexity}
\label{sec:convergence}

In this section, we state two   iteration complexity results for Algorithm \ref{alg:distributedEx}. Theorem \ref{thm:complexity-convex-case} deals with a non-strongly convex objective and shows that
the algorithm achieves sub-linear rate of convergence $\mathcal{O}(\frac1\epsilon)$. Theorem \ref{thm:complexity-strongly-convex-case} shows Algorithm \ref{alg:distributedEx} achieves linear convergence rate $\mathcal{O}(\log\frac1\epsilon)$ for a strongly convex objective. 

However, we wish to stress that in high dimensional settings, and especially in applications where low- or medium-accuracy solutions are acceptable, the dependence of the method on $\epsilon$ is somewhat less important than its dependence on data size through quantities such as the dimension $N$ and the number of blocks $n$, and on quantities such as the number of computers $\K$ and number of parallel updates per computer $\tau$, which is related to the number of cores.

Notice that once the ESO is established by Theorem~\ref{thm:esoGeneral}, the complexity results, Theorems \ref{thm:complexity-convex-case} and \ref{thm:complexity-strongly-convex-case}, follow from the generic complexity results in \cite{rachael:Improved} and  \cite{RT:PCDM}, respectively.





\subsection{Convex functions}

\begin{theorem}[Based on \cite{rachael:Improved}] \label{thm:complexity-convex-case}
Let $f$ satisfy Assumption~\ref{ass:f} and sampling $\hat{\DS}$ satisfy Assumption~\ref{asm:partitions}. Let $x_k$ be the iterates of Algorithm~\ref{alg:distributedEx} applied to problem \eqref{eq:P}, where parameters $\beta$ and $w$ are chosen as in Theorem~\ref{thm:esoGeneral} and the random sets $Z_k$ are iid, following the law of $\hat{\DS}$. Then for all $k\geq 1$,
\begin{align}
\E\left[F(x_k)-F^*\right] 
  \leq \frac{n}{n+ \K \tau k } 
   \left(
   \frac{\beta}{2} \|x_0-x^*\|_w^2
     + F(x_0)-F^*   
   \right).
\end{align}
\end{theorem}

Note that the leading term in the bound decreases as the number of blocks updated in a single (parallel) iteration, $\K \tau$, increases. However, notice that the parameter $\beta$ also depends on $\K$ and $\tau$. We shall investigate this phenomenon in Section~\ref{sec:parall-speedup} and show that the level of speed-up one gets by increasing $\K$ and/or $\tau$ (where by speed-up we mean  the decrease of the upper bound established by the theorem) depends on the degree of separability $\omega$ of $f$. The smaller $\omega$ is, the more speed-up one obtains.

\subsection{Strongly-convex functions}

If we assume that $F$ is strongly convex with respect to the norm $\|\cdot\|_w$ then the following theorem shows that $F(x_k)$ converges to $F^*$ linearly, with high probability. 

\begin{definition}[Strong convexity] Function $\phi:\R^N\to \R\cup \{+\infty\}$
is strongly convex with respect to the norm $\|\cdot\|_w$ with convexity parameter $\mu_{\phi}(w) \geq 0$
if 
\begin{equation}\label{eq:strong_def}\phi(y)\geq \phi(x) + \ve{\phi'(x)}{y-x} + \tfrac{\mu_{\phi}(w)}{2}\|y-x\|_w^2, \quad \forall  x,y \in \dom \phi, \end{equation}
where $\phi'(x)$ is any subgradient of $\phi$ at $x$. 
\end{definition}

Notice that by setting $\mu_\phi(w)=0$, one obtains the usual notion of convexity. 
Strong convexity of $F$ may come from $f$ or $\cPsi$ (or both); we write $\mu_f(w)$ (resp.\ $\mu_\cPsi(w)$) for the (strong) convexity parameter of $f$ (resp.\ $\cPsi$). It follows from \eqref{eq:strong_def} that 
if $f$ and $\cPsi$ are strongly convex, then $F$ is strongly convex with, e.g.,
$\mu_{F}(w) \geq \mu_{f}(w)+ \mu_{\cPsi}(w).$

\begin{theorem}[Based on \cite{RT:PCDM}] \label{thm:complexity-strongly-convex-case}
Let us adopt the same assumptions as in Theorem~\ref{thm:complexity-convex-case}. Moreover, assume that $F$ is strongly convex with $\mu_f(w)+\mu_\cPsi(w)>0$. Choose initial point $x_0\in \R^\N$, target confidence level $0<\rho<1$, target accuracy level $0<\epsilon<F(x_0)-F^*$ and
\begin{equation}\label{eq:k_uniform_strong}
 K\geq
  \frac{n}{\K \tau} \frac{\beta+\mu_\cPsi(w)}{\mu_f(w)+\mu_\cPsi(w)} \log \left(\frac{F(x_0)-F^*}{\epsilon\rho}\right).
\end{equation}
If $\{x_k\}$ are the random points generated by Algorithm \ref{alg:distributedEx}, then
$\Prob(F(x_K)-F^*\leq \epsilon) \geq 1-\rho$.
\end{theorem}

Notice that now both $\epsilon$ and $\rho$ appear inside a logarithm. Hence, it is easy to obtain accurate solutions with high probability. 

\subsection{Parallelization speed-up is governed by sparsity} \label{sec:parall-speedup}

If we assume that $\|x_0 - x^*\|_w^2 \gg F(x_0)-F^*$, 
then in view of Theorem~\ref{thm:complexity-convex-case}, the number of iterations required by our method to get an $\epsilon$ solution in expectation is $O( \frac{\beta}{\K \tau \epsilon})$. Hence, the smaller  $\frac{\beta}{\K \tau \epsilon}$ is, the fewer are the iterations required.  If $\beta$ were a constant independent of $\K$ and $\TT$, one would achieve linear speed-up by increasing workload (i.e., by increasing $\K \tau$).  However, this is the case for $\K=1$ and $\omega=1$ only (see Theorem~\ref{thm:esoGeneral}). Let us look at the general case.
If we write $\eta\eqdef \frac{\xi}{s}$ (this a measure of sparsity of the partitioned data),  then
\begin{align*}
 \frac{\beta}{\K \tau}
  &\overset{\eqref{eq:betaForGeneralCase}}{=}
  \frac{1+\frac{(\xi-1)(\TT-1)}{\max\{1,s-1\}} + (\K-1)\frac{\xi \TT}{s}}{\K\TT}
  \leq
  \frac{1+\frac{\xi(\TT-1)}{s} + (\K-1)\frac{\xi \TT}{s}}{\K\TT}
  \\&\  =
  \frac{1+\eta (\TT-1) + (\K-1) \eta  \TT}{\K\TT}
  =
  \frac{1+\eta    (\K \TT-1) }{\K\TT}
  =\frac{1}{\K\TT}+
  \eta \left(1-\frac{1}{\K\TT}\right).
\end{align*}
As expected, the first term represents linear speed-up. The second term represents a penalty for the lack of sparsity (correlations) in the data. As $\K\TT$ increases, the second term becomes increasingly dominant, and hence slows the speed-up from almost linear to none.
Notice that for fixed $\eta$, the ratio $\frac{\beta}{\K\tau}$ as a function of $\K\TT$ is decreasing and hence we always get {\em some} speed-up by increasing $\K \tau$.

Figure \ref{fig:speed-up} (left) shows the speed-up factor ($\frac{\K \tau }{\beta}$; high values are good)
as a function of $\K\TT$ for different sparsity levels $\eta$. One can observe that sparse problems achieve almost linear speed-up even for bigger value of $\K\TT$, whereas for, e.g., $\eta=0.2$, almost linear speed-up is possible only up to $\K\TT=10$. For  sparser data with $\eta=0.01$, linear speed-up can be achieved up to $\K\TT=100$. For $\eta=0.001$, we can use $\K\TT=10^3$.
The right part of Figure~\ref{fig:speed-up} shows how sparsity affects speed-up for a fixed number of updates $\K \TT$. Again, the break-point of almost linear speed-up is visibly present.
\begin{figure}[tp]
 \centering
 \includegraphics[width=2in]{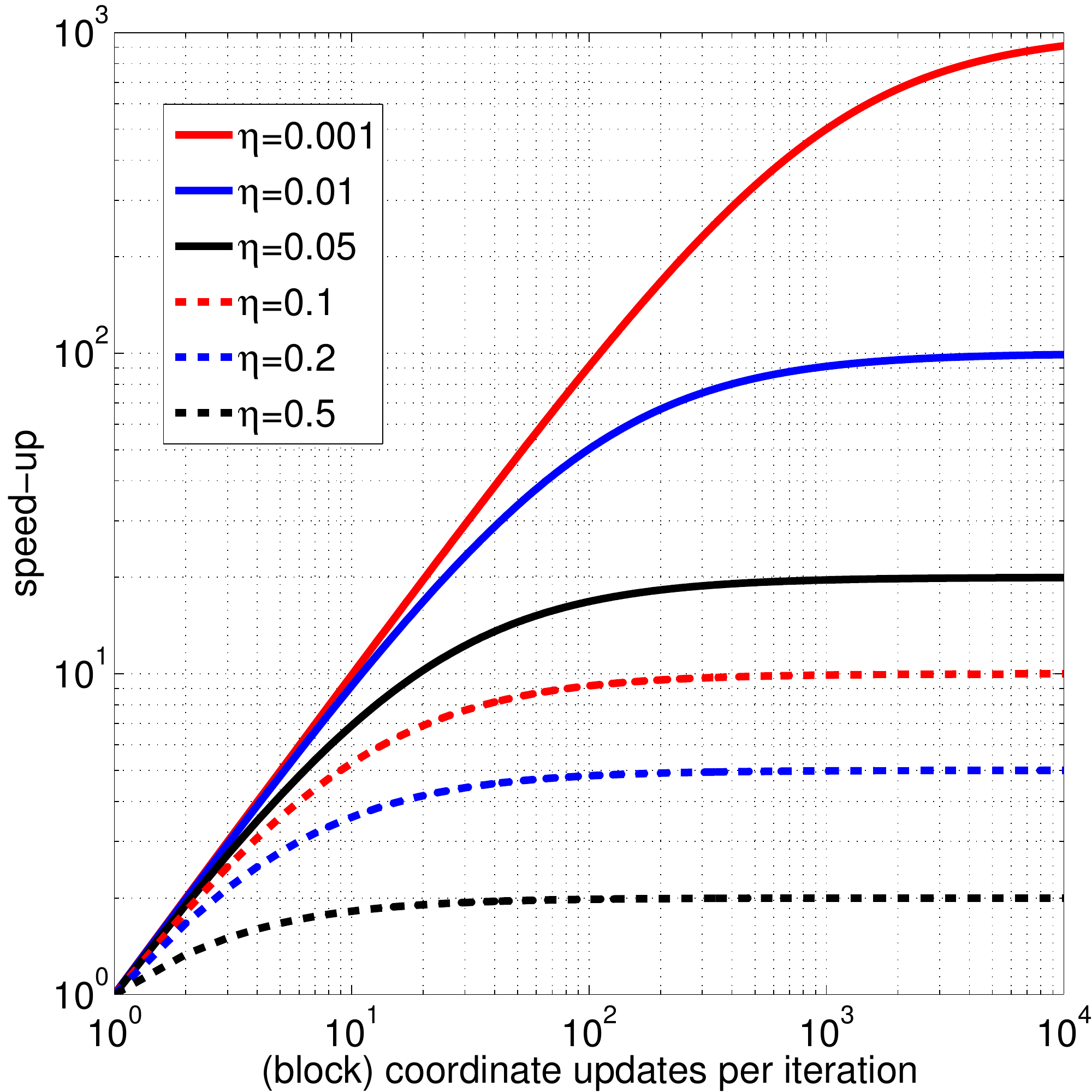}
 \includegraphics[width=2in]{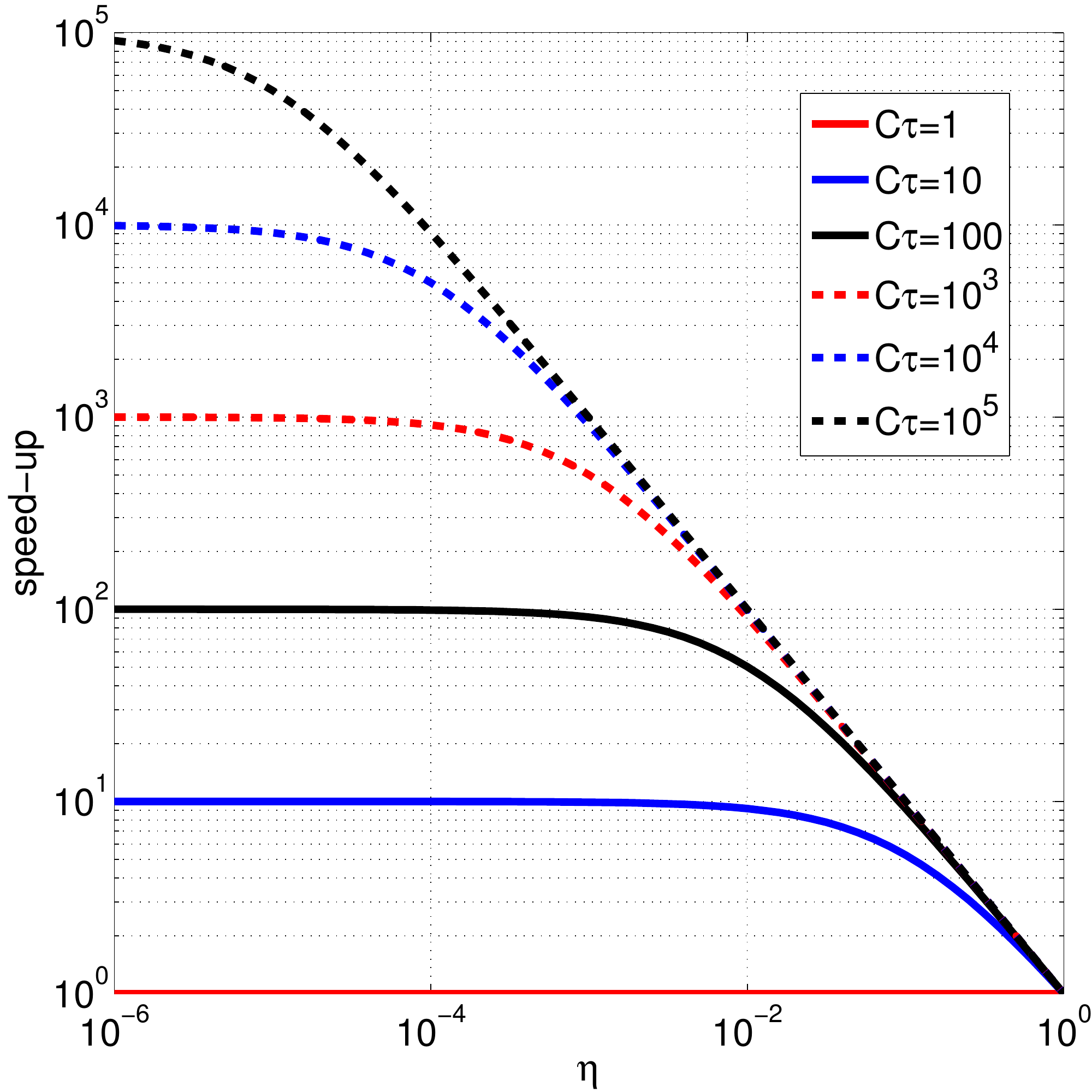}
 \caption{Speed-up gained from updating more blocks per iteration is almost linear initially, and
 depending on sparsity level $\eta$, may become significantly sublinear.}
 \label{fig:speed-up}
\end{figure}

Similar observations in the non-distributed setting were reported in \cite{RT:PCDM}. The phenomenon is not merely a by-product of our theoretical analysis; it also appears in practice. 

\subsection{The cost of distribution}

Notice that in a certain intuitive sense, variants of Algorithm~\ref{alg:distributedEx} are comparable,
as long as each iteration updates the same number $\K\tau$ of blocks.
 This allows us to vary $\K$ and $\tau$, while keeping the product constant. In particular, let us consider two scenarios: 
\begin{enumerate}
\item Consider $\K$ computers, each updating $\tau$ blocks in parallel, and 
\item Consider $1$ computer updating $\K \tau$ blocks in each iteration in parallel.
\end{enumerate} 

For the sake of comparison, we assume that the underlying problem is small enough so that it can be stored on and solved by a single computer. 
Further, we assume that $F$ is strongly convex, $\mu(\cPsi)=0$ and $s = \tfrac{n}{\K} \geq 2$. Similar comparisons can be made in other settings as well, but given the page restrictions, we restrict ourselves to this case only. 

In the iteration-complexity bound \eqref{eq:k_uniform_strong}, we notice that the only difference is in the value of $\beta$. Let $\beta_1$ be the $\beta$ parameter in the first situation with $\K$ computers, and $\beta_2$ be the $\beta$ parameter in the second situation with 1 computer. The ratio of the complexity bounds \eqref{eq:k_uniform_strong} is hence equal to the ratio 
$$
\tfrac{\beta_1}{\beta_2}
 = 
 \frac{
 (1 +\tfrac{(\xi-1)(\TT-1)}{s-1}
 +    (\K-1)
    \tfrac{\xi \TT}{s} ) } {1 +\tfrac{(\omega-1)(\K\TT-1)}{\K s -1}}.
$$
Notice that $\tfrac{\omega}{\K} \leq \xi \leq \omega$. The ratio $\beta_1/\beta_2$ is increasing in $\xi$. We thus obtain the following bounds:
$$
\mbox{LB}:=
\frac
   {1 +\frac{(\omega-\K)(\TT-1)}{n-\K}
 +    (\K-1)
    \frac{\omega \TT}{n}     }
  {1 +\frac{(\omega-1)(\K\TT-1)}{n-1}}
\leq
\frac{\beta_1}{\beta_2}
 \leq \frac
   {1 +\frac{(\omega-1)(\K\TT-\K)}{n-\K}
 +    (\K-1)
    \frac{\omega \K \TT}{n}     }
  {1 +\frac{(\omega-1)(\K\TT-1)}{n-1}} =: \mbox{UB}.
$$

Table \ref{tbl:boundOnBetaForNiceVSDistr}
presents the values of $\mbox{LB}$ and $\mbox{UB}$ for various parameter choices and problem sizes. We observe that the value of $\beta_2$ is around 1.
The value of $\beta_1$ depends on a particular partition, but we are sure that $\beta_1 \in [\beta_2 \cdot \mbox{LU}, \beta_2 \cdot \mbox{UB}]$.
In Table \ref{tbl:boundOnBetaForNiceVSDistr}, $\mbox{UB}$ is less than 2, which means that by distributing the computation, the method will at most double the number of iterations. However, larger values of UB,
 albeit $\mbox{UB} \precsim \K$,
are possible for different settings of the parameters.
For a different class of functions $f$, an upper bound of 2 was proven in \cite{richtarik2013distributed} and improved in \cite{fercoq2014fast} to the factor $1+1/(\tau-1)$ whenever $\tau>1$.

Of course, if the  problem size exceeds the memory available at a single computer, the option of not distributing the data and computation may not be  available. 
It is reassuring, though, to know that the price we pay for distributing the data and computation, in terms of the number of iterations, is bounded. 
Having said that, a major complication associated with any distributed method is the communication, which we discuss in the two following sections.

\begin{table}[htb!]

 \caption{Lower and upper bounds on $\beta_1/\beta_2$ for a selection parameters $n,\omega,\K$ and $\tau$.}
 \label{tbl:boundOnBetaForNiceVSDistr}

 \centering
 \begin{tabular}{r|r|r|r|r|r|r}
 \multicolumn{1}{c|}{$n$} &
 \multicolumn{1}{|c|}{$\omega$} & \multicolumn{1}{|c|}{$\K$}  & \multicolumn{1}{|c||}{$\TT$} & \multicolumn{1}{|c|}{$\beta_2$} & \multicolumn{1}{|c|}{LB} & \multicolumn{1}{|c}{UB}
 \\ \hline \hline
 $10^6$&$10^2$ & 10 & $50$& 1.049 & 1.0000086 &   1.4279673
 \\
 $10^7$&$10^2$ & 10 & $50$& 1.005 & 1.0000009 &   1.0446901
 \\
 $10^8$&$10^2$ & 100 & $100$& 1.009 & 1.0000010 & 1.9801990
 \end{tabular}
\end{table}

\section{Two implementations}\label{sec:implementation}

Although our algorithm and results apply to a rather broad class of functions, we focus on two important problems
in statistics and machine learning in describing our computational experience,
so as to highlight the finer details of the implementations.



\subsection{An implementation for sparse least squares}
\label{app:SLS}

In many statistical analyses, e.g., linear regression, one hopes to find a solution $x$ with only a few non-zero elements, which improves interpretability.
 It has been recognized, however, that the inclusion of the number of non-zero elements, $\|x\|_0$, in the objective function
raises the complexity of many efficiently solvable problems to NP-Hard \cite{MR1320206,MR2837883}. Recently, a number of randomized coordinate descent methods try to handle the $\ell_0$-norm directly \cite{patrascu2014random}, but only local convergence can be guaranteed.
Fortunately, the inclusion of the sum of absolute values, $\|x\|_1$, provides a provably good proxy, which is also known as $\ell_1$ regularization.
There is a large and growing body of work on both practical solvers for non-smooth convex problems, obtained by such a regularization, and their convergence properties,
when one restricts oneself to a single computer storing the complete input.
Such solvers are, however, most useful in high-dimensional applications,
where the size of the data sets often exceeds the capacity of random-access memory of any single computer available today.

Hence, the first implementation we present is a distributed coordinate-descent algorithm for $\ell_1$-regularized (``sparse'') least squares. The key components needed by  Algorithm \ref{alg:distributedEx} are the
computation of $\Lip_i$, $\nabla_i f(x_k)$, and solving of a block-wise minimization problem.
Note that
$ \nabla_i f(x) =  \sum_{j=1}^m  - \mel{A}{j}{i} (\vc{y}{j} - \mrow{A}{j} x)$,
where $\mrow{A}{j}$ denotes $j$-th row of matrix $A$,
 and $L_i = \|\mcol{A}{i}\|_2^2.$
 The only difficulty is that given the data partition $\{\vc{\Partc}{\Kidx}\}_{\Kidx=1}^\K$, no single computer $\Kidx$ is able to compute $\nabla_i f(x)$ for any $i\in \vc{\Partc}{\Kidx}$.
The reasoning follows from a simple observation: if we wanted to compute $\nabla_i f(\vt{x}{k})$ for a given $\vt{x}{k}$ from  scratch, we would have to
access all coordinates of $\vt{x}{k}$, vector $y$, and all non-zero elements of the input matrix $A$.
This could be avoided by introducing an auxiliary vector  $\vt{g}{k}:=g(\vt{x}{k})$ defined as
\begin{align}
\label{eqn:residuals}
\vt{g}{k} \eqdef
 \; A \vt{x}{k} - y.
\end{align}
Once the value of $g_k=g(x_k)$ is available, a new iterate is 
\begin{equation}
\vt{x}{k+1} = \vt{x}{k} + \sum_{\Kidx=1}^\K \sum_{i \in \Zj{\Kidx}{k}} U_i \vc{h}{i}(\vt{x}{k}).
\end{equation}
and $g_{k+1}=g(x_{k+1})$ can be easily expressed as
\begin{align}
\vt{g}{k + 1}
&=
\vt{g}{k} +\sum_{\Kidx=1}^\K \underbrace{\sum_{i \in \Zj{\Kidx}{k}} \mcol{A}{i} \bvc{h}{i}(\vt{x}{k})}_{\vc{\delta g}{\Kidx}}.
\end{align}
Note that the value $\vc{\delta g}{\Kidx}$ can be computed on computer $\Kidx$ as all required data are available on computer $\Kidx$.
Subsequently, $g_{k+1}$ can be obtained by summation and the formula for $\nabla_i f(x)$ will take the form
$
\nabla_i f(x) =  
   \mcol{A}{i}^T g = \sum_{j=1}^m   \mel{A}{j}{i} \vc{g}{j}.$
Once we know how to compute $\nabla_if(x)$ and $\Lip_i$, all that remains to be done is to solve the 
problem
\begin{equation}
 \min_{t \in \R} a + b t + \frac{c}{2} t^2 + \lambda |d+t|,
\end{equation}
where $a,b,d\in \R$ and $c,\lambda\in\R_{>0}$, which is given by a {\it soft-thresholding} formula
$
 t^*
 =  \mysgn(\zeta) (|\zeta|-\tfrac \lambda{c})_+ - d,
$
where
$\zeta=d-\tfrac{b}{c}.$

\subsection{An implementation for training support vector machines} 
\label{app:SVM}

Let us present another example implementation.
The key problem in supervised machine learning is the training of classifiers.
Given a matrix $A\in\R^{m \times N}$, a compatible vector $y \in \R^m$, and constant $\gamma > 0$,
the goal is to find a vector $x\in\R^N$ which solves the following optimization problem:
\begin{equation}\label{eq:SVMproblemFormulation}
  \min_{x\in\R^N} F(x) \eqdef \underbrace{\gamma \|x\|_1}_{\cPsi(x)} + \underbrace{\sum_{j=1}^m \lf(x, \mrow{A}{j}, \vc{y}{j})}_{f(x)},
\end{equation}
where $\mrow{A}{j}$ again denotes $j$-th row of matrix $A$ and $\lf$ is a loss function,
such as
\begin{align}
\lf_{SL}(x, \mrow{A}{j}, \vc{y}{j}) \eqdef & \frac12 (\vc{y}{j} - \mrow{A}{j} x )^2, & \mbox{ square loss}, \label{SL} \tag{SL} \\
\lf_{LL}(x, \mrow{A}{j}, \vc{y}{j}) \eqdef & \log (1+e^{- \vc{y}{j} \mrow{A}{j} x } ), & \mbox{ logistic loss}, \label{LG} \tag{LL} \\
\lf_{HL}(x, \mrow{A}{j}, \vc{y}{j}) \eqdef & \frac12 \max\{0, 1 - \vc{y}{j} \mrow{A}{j} x \}^2, & \mbox{ hinge square loss}. \label{L2-SVM} \tag{HL}
\end{align}
The input $(A, y)$ is often referred to as the training data.
Rows of matrix $A$ represent observations of $N$ features each
and $y$ are the corresponding classifications to train the classifier on.


Square hinge loss is a popular choice of $\lf$, but is not smooth. It is well known that the dual has the form \cite{DCD,shalev2013stochastic,TR:MINIBATCH}:
\begin{equation}\label{eq:SVM-DUAL}\tag{SVM-DUAL}
  \min_{x \in\R^m} F(x) \eqdef
    \underbrace{\frac{1}{2\lambda m^2} x^T Q x -
    \frac1m x^T {\bf 1}}_{f(x)} +
      \underbrace{\sum_{i=1}^m \indic_{[0,1]}(\vc{x}{i})}_{\cPsi(x)},
\end{equation}
where
$\indic_{[0,1]}$ is the characteristic (or ``indicator'') function of the interval $[0,1]$
and $Q\in \R^{m \times m}$ is the Gram matrix of the data, i.e.,
$Q_{i,j} = \vc{y}{i} \vc{y}{j}   \mrow{A}{i}  \mrow{A}{j}^T $.
If $x^*$ is an optimal solution of \eqref{eq:SVM-DUAL} then
$w^* = w^*(x^*)=\frac1{\lambda m} \sum_{i=1}^m \vc{y}{i} \vc{(x^*)}{i} \mrow{A}{i}^T$
is an optimal solution of the primal problem
\begin{equation}\label{eq:HingeLoss}
  \min_{w \in\R^N} P(w) \eqdef \frac1N \sum_{i=1}^N
  \lf(w, \mrow{A}{i}, \vc{y}{i}) + \frac{\lambda}2 \|w\|^2,
\end{equation}
where $\lf(w, \mrow{A}{i}, \vc{y}{i})=\max\{0,  1 - \vc{y}{i} \mrow{A}{i} w \}$.

Our second example implementation is a distributed coordinate-descent algorithm for support vector machines (SVM) in the \eqref{eq:SVM-DUAL} formulation.
In this case, we define
\begin{equation}
 g_k := \frac1{\lambda m} \sum_{i=1}^m \vc{x_k}{i} \vc{y}{i} \mrow{A}{i}^T.
\end{equation}
Then
\begin{equation}
 \nabla_i f(x)=  \frac{\vc{y}{i} \mrow{A}{i} g_k - 1}{ m}
 , \qquad \Lip_i = \frac{\|\mrow{A}{i}\|^2}{\lambda m^2}.
\end{equation}
The optimal step length is then solution of a one-dimensional problem:
\begin{align}
 \vc{h}{i}(x_k) & = \arg\min_{t\in\R} \nabla_i f(\alpha) t + \frac {\beta} 2 \Lip_i t^2 +
   \indic_{[0,1]} ( \vc{\alpha}{i}+t ) \\
   & = \clip_{[-\vc{\alpha}{i}, 1-\vc{\alpha}{i}]}\left( \frac{\lambda m (1-\vc{y}{i} \mrow{A}{i} g_k )}{\beta \|\mrow{A}{i}\|^2}  \right),
\end{align}
where for $a < b$
$$\clip_{[a,b]}(\zeta)
 = \begin{cases}
    a,& \mbox{if} \ \zeta <a,\\
    b,& \mbox{if} \ \zeta > b,\\
    \zeta,& \mbox{otherwise}.
   \end{cases}
$$

The new value of the auxiliary vector
$g_{k+1}=g(x_{k+1})$ is given by
\begin{align}
\vt{g}{k + 1}
 =
\vt{g}{k} +\sum_{\Kidx=1}^\K
\underbrace{\sum_{i \in \Zj{\Kidx}{k}} \frac1{\lambda m} \bvc{h}{i}(\vt{x}{k}) \vc{y}{i} \mrow{A}{i}^T}_{\vc{\delta g}{\Kidx}}
\end{align}
and the duality gap $G(x_k) = P(g_k) + F(x_k)$ can be easily obtained \cite{TR:MINIBATCH,shalev2013stochastic,jaggi2014communication} as
\begin{align}
G(x_k)&= \frac1m \sum_{i=1}^m
  (\lf(g_k, \mrow{A}{i}, \vc{y}{i})  -\vc{x_k}{i}  )+ \lambda \|g_k\|^2.
\end{align}

\section{Per-iteration complexity}
\label{para:distributed-computation}

Using to the auxiliary vector $\vt{g}{k}$, which was introduced in the previous section, Algorithm \ref{alg:distributedEx} has two alternating and time consuming sub-procedures, namely:
\begin{enumerate}
 \item computation of an update
$\sum_{i \in \Zj{\Kidx}{k}} U_i \vc{h}{i}(\vt{x}{k})$ and 
the accumulation of $\vt{g}{k}$: $\vc{\delta g}{\Kidx}$,  
\item updating $\vt{g}{k}$ to $\vt{g}{k+1}$.
\end{enumerate}
Let us denote the run-time of the first sub-procedure by
$\RT_1(\TT)$, considering this depends on $\TT$,
and the run-time of a second one by $\RT_2$.
We will neglect the rest of the run-time cost, such as managing a loop, evaluation of termination criteria, measuring a computation time, etc.
The total run-time cost $\RT_T$ is hence given by
\begin{equation}\label{eq:totalRT}
 \RT_T = \BO{ \frac{\beta}{\K\TT} (\RT_1(\TT) + \RT_2)   }
\end{equation}
where we consider the case when $\mu_\cPsi(w)\equiv 0$ in
\eqref{eq:k_uniform_strong}.
 Let us now for simplicity assume that the first sub-procedure is linear in $\TT$, i.e.,
$\RT_1(\TT) = \TT \RT_1(1) =:\TT \RT_1$.
Then
\begin{equation}\label{eq:totalRT2}
 \RT_T = \BO{ \frac{\beta}{\K\TT} (\TT \RT_1  + \RT_2)   }.
\end{equation}
Numerical values of $\RT_1$ and $\RT_2$ could be estimated, given problem sparsity and underlying hardware, or can be measured during the run.

{\bf Optimal choice of sampling parameter $\TT$.}
In the previous paragraph, we gave an estimate of the complexity of a single iteration. In this paragraph,
we answer the question of how to choose a $\TT$ given times $\RT_1, \RT_2$. For variable $\beta$, we have more options,
but we stick to the most general one given in \eqref{eq:betaForGeneralCase}. Given that $s \geq 2$, we have
\begin{align}\label{eq:asvfwiuj32f2q}
\RT_T &= \BO{  \tfrac{1 +\frac{(\xi-1)(\TT-1)}{s-1}
 +    (\K-1)
   \frac{\xi \TT}{s }}{\K} \left( r_{1,2}  +\frac1{\TT} \right) \RT_2}
=
\BO{   \left(\frac{s}{\xi\K} + \TT  \right)
 \left( r_{1,2}  +\frac1{\TT} \right)},
\end{align}
 where $r_{1,2}=\frac{\RT_1}{\RT_2}$ is a work to communication ratio.
 The optimal parameter $\TT^*$ can be obtain by minimizing \eqref{eq:asvfwiuj32f2q} and is given by
 \begin{equation}
\label{eq:optimalTTParameter}
\TT^* = \sqrt{\frac{s}{r_{1,2} \ \xi\ \K }}.
 \end{equation}
Therefore, smaller values of $r_{1,2}$ imply that we should do more work in each iteration, and hence bigger values of $\TT$ should be chosen.
This is quite natural, as one should tune the parameters in such a way that time spent in communication should be in comparable with that of effective computation.

\begin{figure}[tp]
 \centering
 \includegraphics[width=4in]{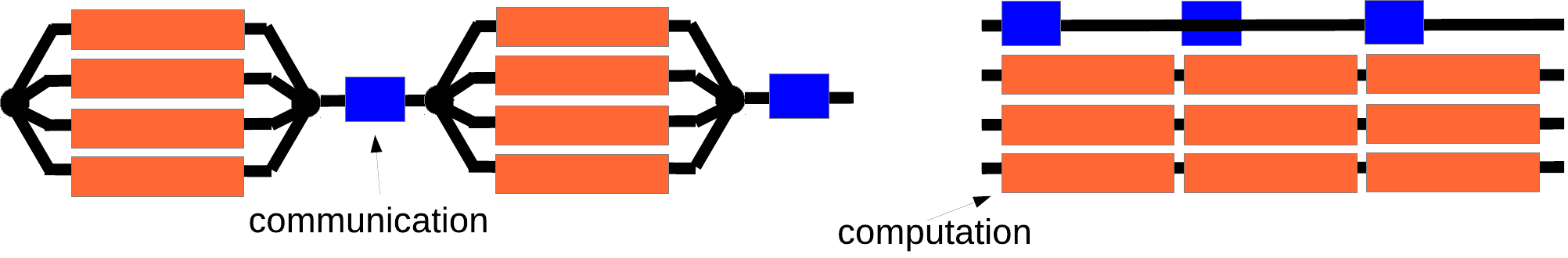}
 \caption{An illustration of a na\"{i}ve (PS) approach (left), which alternates between parallel regions, where computations take place, and serial regions dedicated to MPI communications with other computers. An alternative (FP) approach (right) dedicates the communication task to one thread and uses other threads for computation.}
 \label{fig:serialBlocksAndParallel}
\end{figure}

{\bf Message Passing Interface (MPI).}
In order to discuss finer details of the implementations, we need to introduce 
the architecture we use. We use OpenMP \cite{OpenMP} for dealing with concurrency within a single computer and Message Passing Interface (MPI) \cite{MPI} as the abstraction layer for network communication.
In MPI, one passes data from one MPI process to another MPI process, which
may run on another computer. (We disregard the concept of groups for brevity.)
Communication can involve any subset of computers, which run MPI processes.
Communication can be either blocking (``synchronous'') or non-blocking (``asynchronous''). 
A \emph{collective} operation
involves the communication among two or more MPI processes. An example of a collective operation is a {\it barrier}, where computers wait until all of them reach the same point in the algorithm.
Another common collective operation is {\it reduce all},
which is parametrized by an arbitrary operation that takes a set of elements and produces a single element of the same type.
This ``reduce'' operation is applied to all elements of the particular type stored across all MPI processes and the result is returned to all MPI processes.
For example, let us assume that each computer stores a vector $\vc{\delta g}{\Kidx}\in\R^m$ and the goal is to sum it up, i.e., to compute $\vc{\delta g}{1,\dots,\K}=\sum_{\Kidx=1}^\K \vc{\delta g}{\Kidx}$ and to make this result available on each computer.
Figure \ref{fig:reduceAll} shows a standard approach, which leads to the desired result.
From the performance point of view, however, the use of {\it reduce all} should be minimized,
as it involves an implicit synchronisation and leaves most of the computers idle throughout the collective operation.

\begin{figure}[tp]
 \centering
 \includegraphics[width=2.5in]{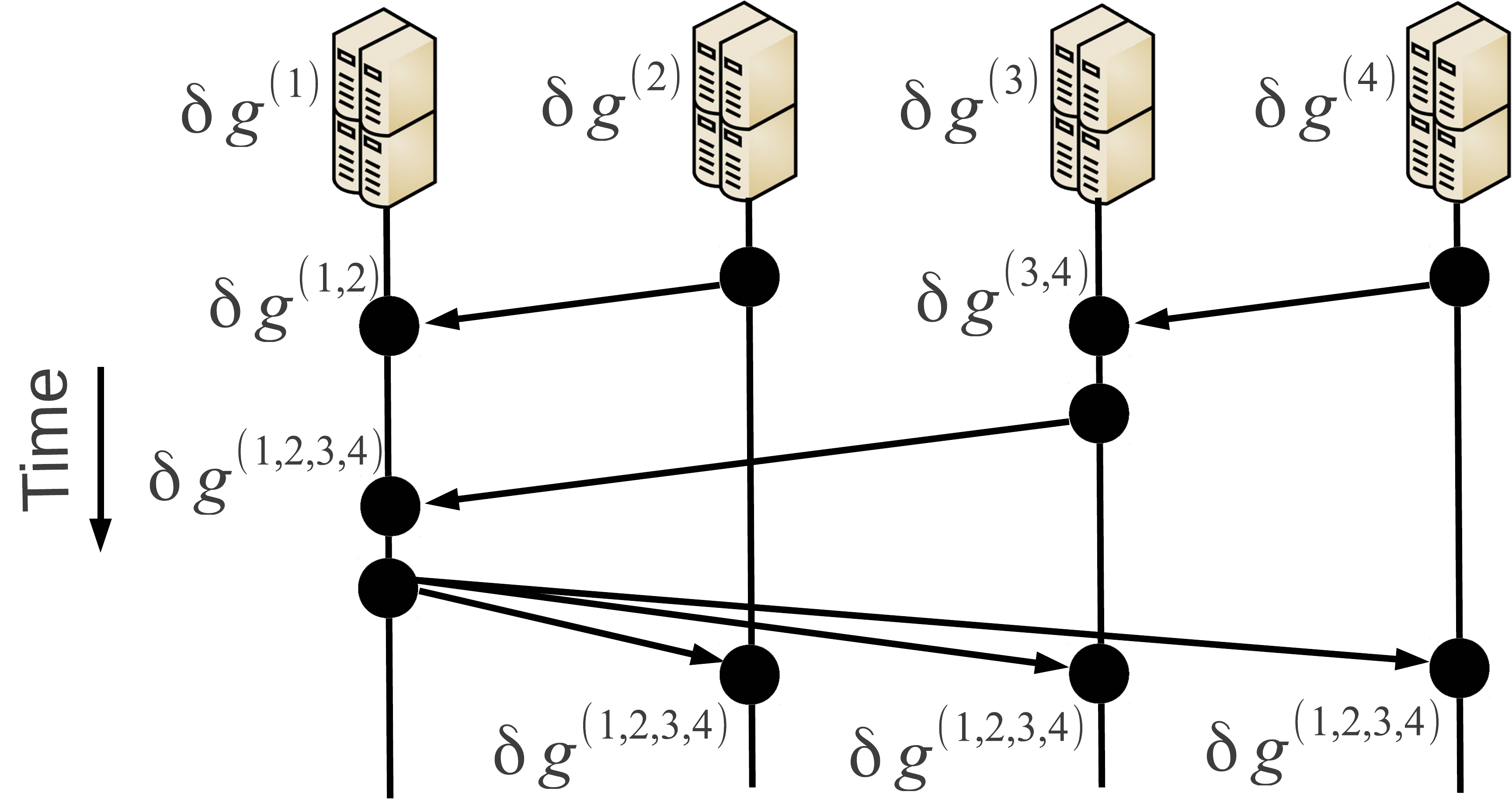}
 \caption{Schematic diagram of a standard reduce all implementation. The goal is to compute $\sum_{\Kidx=1}^{\K} \vc{\delta g}{\Kidx}$.  The arrows show data flow between computers.}
 \label{fig:reduceAll}
\end{figure}

This suggests the following range of progressively better-performing variants:

{\bf Alternating Parallel and Serial regions (PS).}
The na\"{i}ve implementation alternates two sub-procedures. One, which is computationally heavy and is done in parallel, but with no MPI communication, and another one, which is purely communicational. As an easy fix, one can dedicate one thread to the communication and other threads within the same computer to computation. We call this approach {\bf Fully Parallel} (FP)). Figure \ref{fig:serialBlocksAndParallel} compares the na\"{i}ve strategy (left) with the FP (right),

{\bf Reduce All (RA).} As mentioned above, the use of {\it reduce all} operations significantly decreases the performance of many distributed algorithms. It is, however, the preferred form of communication between computers close to each other in the computer network, such as computers directly connected by a network cable.
The use of asynchronous methods is also preferred over synchronous methods.

{\bf Asynchronous StreamLined (ASL).} We propose another pattern of communication, where each computer in one iteration sends only one message to the closest computer, asynchronously, and receives only one message  from another computer close-by, asynchronously. The communication hence takes place in an ring. This tweak, however, requires a significant change in the algorithm.
 Figure \ref{fig:ASL} illustrates the data flow of messages at the end of  iteration $k$ for $\K=4$.
 \begin{figure}[tp]
  \sidecaption
 \includegraphics[width=7.3cm]{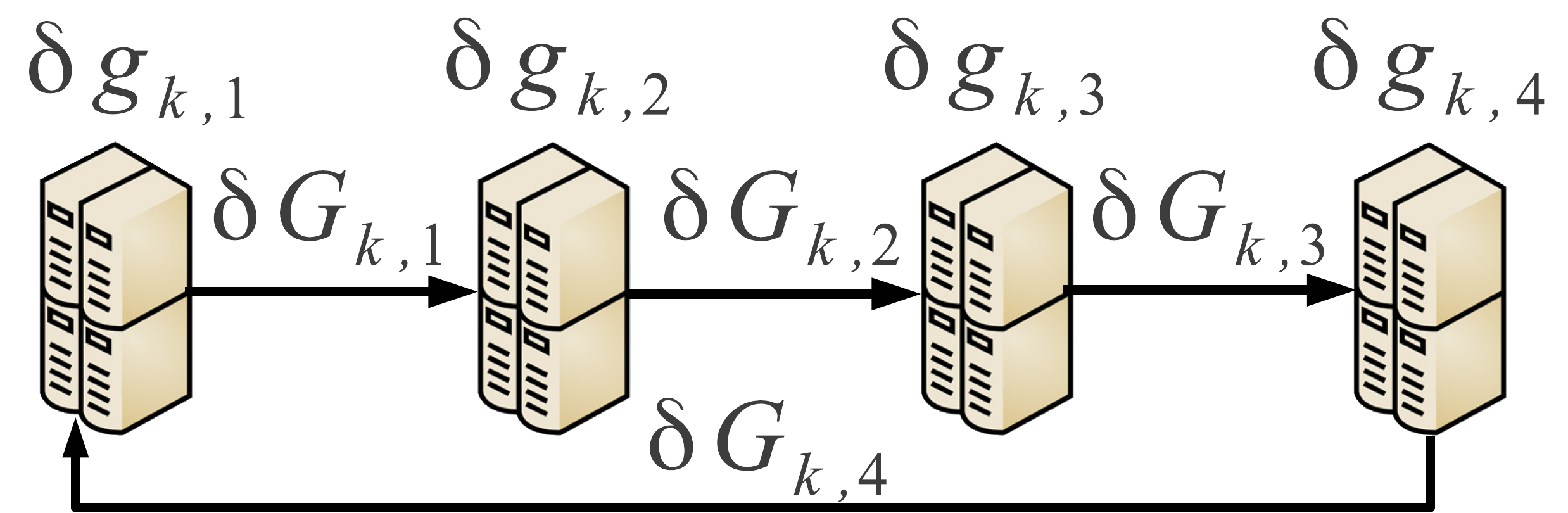}
 \caption{Illustration of ASL method for $\K=4$.
 During $k$-th iteration, computer $\Kidx$ obtains its contribution $\vc{\delta g_k}{\Kidx}$
 but asynchronically sends an accumulated update $\vc{\delta G_k}{\Kidx}$ to its successor.}
 \label{fig:ASL}
\end{figure}
 We fix an order of computers in a ring, denoting $\predecessor_R(\Kidx)$ and $\successor_R(\Kidx)$ the two computers neighbouring computer $\Kidx$
along the two directions on the ring.
Computer $\Kidx$ always receives data only from computer $\predecessor_R(\Kidx)$ and sends data only to computer $\successor_R(\Kidx)$.
Let us denote by $\vc{\delta G_k}{\Kidx}$ the data, which
computer $\Kidx$ sends to computer $\successor_R(\Kidx)$ at the end of iteration $k$.
When computer $\Kidx$ starts iteration $k$, it has already received $\vc{\delta G_{k-1}}{\predecessor_R(\Kidx)}$.\footnote{For the start of the algorithm we define $\vc{\delta g_l}{\Kidx}=\vc{\delta G_{l}}{\Kidx}={\bf 0}$ for all $l < 0$.}
Hence the data, which will be sent at the end of iteration $k$ by computer $\Kidx$ are:
\begin{equation}\label{eq:cumulativeupdate}
 \vc{\delta G_k}{\Kidx} = \vc{\delta G_{k-1}}{\predecessor_R(\Kidx)} - \vc{\delta g_{k-\K}}{\Kidx} + \vc{\delta g_k}{\Kidx}.
\end{equation}
It should be noticed that at the end of each iteration in the ASL procedure, each computer has a different vector $g_k$, which we denote $\vc{g_k}{\Kidx}$. The update rule is 
\begin{equation}
 \vc{g_{k+1}}{\Kidx} = \vc{g_{k}}{\Kidx} + \vc{\delta g_k}{\Kidx}
  + \vc{\delta G_{k}}{\predecessor_R(\Kidx)} - \vc{\delta g_{k-\K+1}}{\Kidx}.
\end{equation}
The clear advantage of the ASL method is a decrease in communication time. On the other hand it comes with a cost of slower propagation of information. Indeed, it takes $\K-1$ iterations to propagate information to all computers. It also comes with bigger storage requirements, as at iteration $k$, we have to have all vectors $\vc{\delta g_l}{\Kidx}$ for $k-\K \leq l \leq k$ stored on computer $\Kidx$.

{\bf \bf Asynchronous Torus (AST).}
There is a compromise solution, though, which inherits many desirable features of both RA and ASL.
This employs a toroidal networking topology, which is common in high-performance computing (HPC) in general, and HPC using InfiniBand networks \cite{InfiniBand}, in particular.
Let us assume that $\K$ is a multiple of $r\in\N$,  where $r$ represents the width of a torus, i.e., $\K$  computers are partitioned into subsets $R_i$ each with size $r$.
Each group $R_i$ has a root computer. These root computers aggregate updates from their respective groups, e.g., using a local {\it reduce all} operation, in each iteration and exchange those update in an asynchronous ring with two other adjacent root computers.
Thus the communication between the root nodes follows the ASL communication pattern. The AST approach decreases the propagation time from $\K$ to $\frac\K r$, additional storage is also decrease by factor $r$, and the overall communication complexity remains low.

{\bf The Comparison.}
Changing from the FP approach to the PS approach does not require much computational or storage overhead, but can reduce the idle time of processors.
However, changing from RA to SLA or AST brings significant storage requirements, while it reduces both communication and idle time significantly.
Table \ref{tbl:runtimeAndStorage} summarize maximum memory requirements on each single node of the cluster, time spent in communication,
and amount of data transferred over the network. Once the time spent in communication is measured or estimated, one can pick the most appropriate strategy.
Notice that the wall-clock time required for the {\it reduce all} operation, $\RT_{ra}$, is 
typically of the order
$\mathcal{O}(\log \K) \cdot \RT_{p2p},$ where $\RT_{p2p}$ is the time required by the point-to-point transmission.
\begin{table}[tp]

 \caption{Summary of additional memory and computation requirements for strategies RA, SLA, AST.}
 \label{tbl:runtimeAndStorage}

 \centering
 \begin{tabular}{c|c|c|c}
  strategy & memory for $g$'s & communication & extra computation
  \\ \hline \hline
  RA & $2m$ & $\RT_{ra}$ & 0
  \\
  SLA & $(2+\K)m$ &  $\RT_{p2p}$ & $4 m$ additions
  \\
  AST & $(2+ \K / r)m$ & $\RT_{p2p}$ + $\RT_{ra}/ r$ & $8 m$ additions
 \end{tabular}
\end{table}

\section{Numerical experiments}
\label{sec:results}

In this  section we present numerical evidence of the efficiency of the  distributed (block) coordinate-descent method. 

{\bf The code.}
The code of the distributed (block) coordinate-descent solver is part of our AC-DC library, available at \url{http://code.google.com/p/ac-dc/}. 
The library is written in C++ using OpenMP. 
The extensive use of template classes, Boost::MPI, and Boost.Serialization makes it easy to change the composite function and the precision of the computation.
Both wall-clock and CPU-time were measured using Boost::Timers,
which achieve nano-second accuracy on recent processors running recent versions of Linux.

{\bf The facility.} 
Our empirical tests were conducted in UK's high-performance computing facility, HECToR, equipped with multi-core computers connected using Infiniband \cite{InfiniBand}.
In particular, in Phase 3 of the facility, which is a Cray XE6 cluster, we have used up to 128 nodes,
 equipped with two AMD Opteron Interlagos 16-core processors and 32 GB of memory each.
This gave us 4,096 cores in total,
interconnected using Cray Gemini routers in a 3D torus.
Each Gemini router was connected to processors and random-access memory 
 of two nodes via HyperTransport links.
Each router is then connected to ten other routers.
In practice, the latency is about 1--1.5 microseconds and
 the capacity of each link is 8~GBs$^{-1}$.
The facility ran a Cray Linux Environment, based on SuSE Linux.


{\bf Support Vector Machines (SVM).}
One of the goals of this paper is to train huge sparse support-vector
machines (SVM) that do not fit into the memory of a single computer. In the machine learning literature, one often performs experiments on instances of moderate size, e.g., 100 MB  \cite{pegasos,DCD,TR:MINIBATCH}.
Well-known instances of this scale include, e.g., CCAT variant of RCV1 \cite{Lewis2004}, Astro-ph \cite{pegasos}, and COV \cite{pegasos}. 
In this Section, we focus on a larger dataset, known as WebSpam \cite{libsvm}. 
This dataset consists of 350,000 observations (rows) and 16,609,143 features (columns).
The size of the instance is 25 GB.
Figure \ref{fig:SVMDistributedresults} show the execution time and duality gap
for WebSpam dataset, using $\K=16$ MPI processes, with each process using 8 threads. 
$\tau$ is the number of coordinates updated by one MPI process during one iteration. As expected, the main run-time cost it not computing the updates, but updating $g$.
 \begin{figure}[tp]
 \sidecaption
\includegraphics[width=6cm]{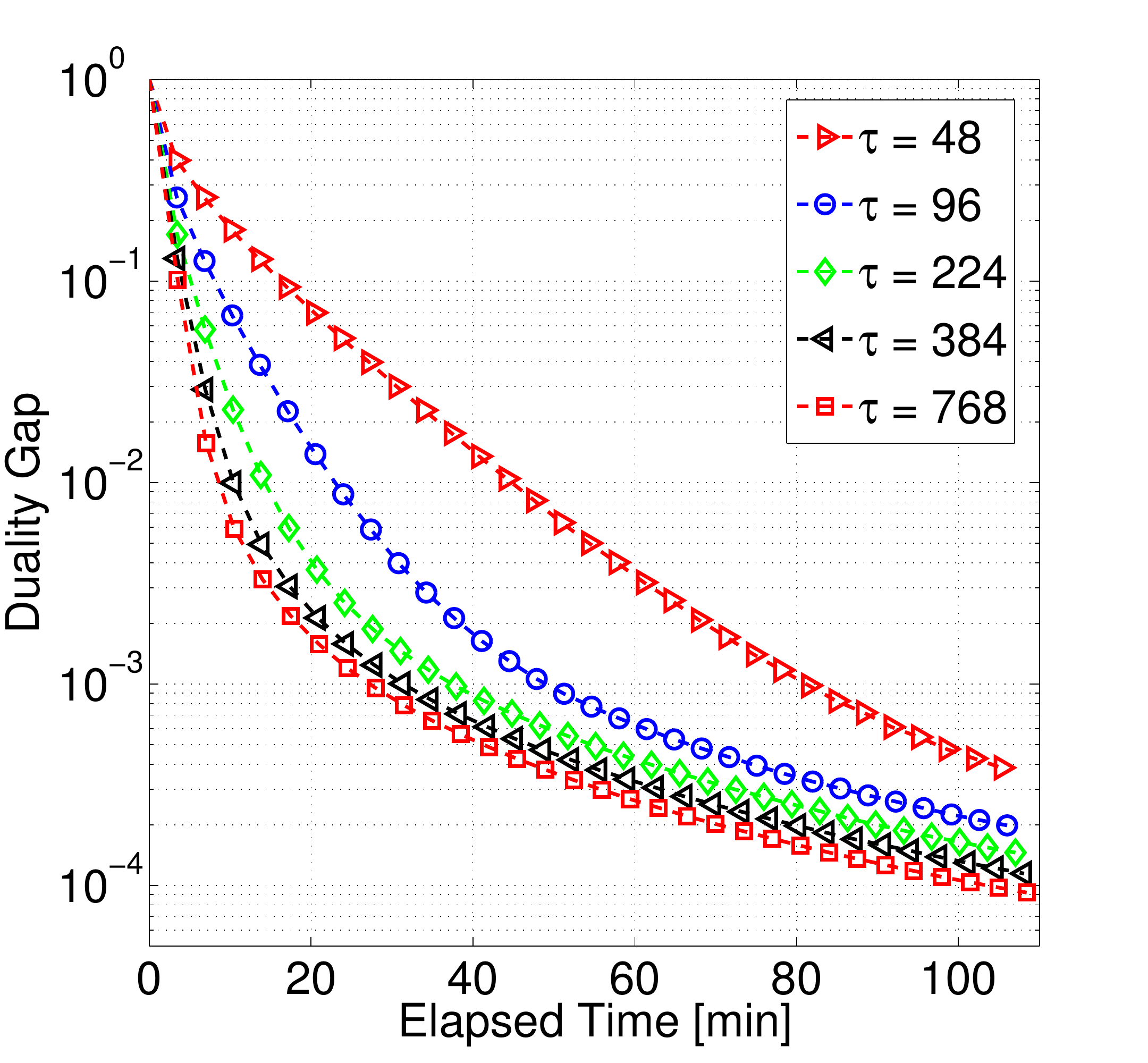}
 \caption{Evolution of duality gap for the WebSpam dataset for various choices of $\tau$.}
 \label{fig:SVMDistributedresults}
 \end{figure}
Let us remark that $\epsilon$ is usually not 
particularly small in the machine-learning community. In experimenting with small $\epsilon$, we just wanted to demonstrate that our algorithm is able to close the duality gap within the limits of machine precision.
The truly important measures of the performance of the classifier, e.g., 0-1 loss or prediction error, are actually within 10 \% after the first minute, which is the first time we compute it. In practice,
 a duality gap of 0.1 or 0.01 can be sufficient for machine learning problems.

{\bf Sparse least squares (LASSO).}
Next, we solved an artificial instance of sparse least squares with a matrix 
of $n = 10^9$ rows and $d = 5 \cdot 10^8$ columns in block-angular form:
\begin{equation}\label{eq:block_ang}
A=\left(
\begin{BMAT}(rc){c;c;c;c}{ccc;c}
A^{(1)}_{loc} & 0   & \cdots & 0\\
0 & A^{(2)}_{loc}   & \cdots & 0\\
\vdots & \vdots    & \ddots & \vdots\\
A_{glob}^{(1)} & A_{glob}^{(2)}   & \cdots & A_{glob}^{(\K)}
\end{BMAT}
\right).
\end{equation}
requiring 3 TB to store.
Such matrices often arise in stochastic optimization. 
We used 128 nodes with 4 MPI processes on each node.
Each MPI process ran 8 OpenMP threads, giving a total of
4,096 hardware threads. 
Each node $\Kidx$ stored two matrices: $A_{loc}^{(\Kidx)} \in \R^{1,952,148 \times 976,562}$
and $A_{glob}^{(\Kidx)} \in \R^{500,224 \times 976,562}$. The average number of non-zero elements per row is $175$ and $1,000$ for 
$A^{(\Kidx)}_{loc}$ and
$A_{glob}^{(\Kidx)}$, respectively. 
When communicating $g_{k}^{(\Kidx)}$, only entries corresponding to the global part of $A^{(\Kidx)}$  need to be communicated, and hence in RA, a {\it reduce all} operation is applied to vectors $\delta g_{glob}^{(\Kidx)} \in \R^{500,224}$. In ASL, vectors with the same length are sent.
The optimal solution $x^*$ has exactly $160,000$ nonzero elements. 
Figure~\ref{fig:ASL_vs_RA} compares the evolution of $F(x_k)-F^*$ for  ASL-FP and RA-FP.

{
\begin{figure}[htp]
 \sidecaption
 \centering
 \includegraphics[width=6cm]{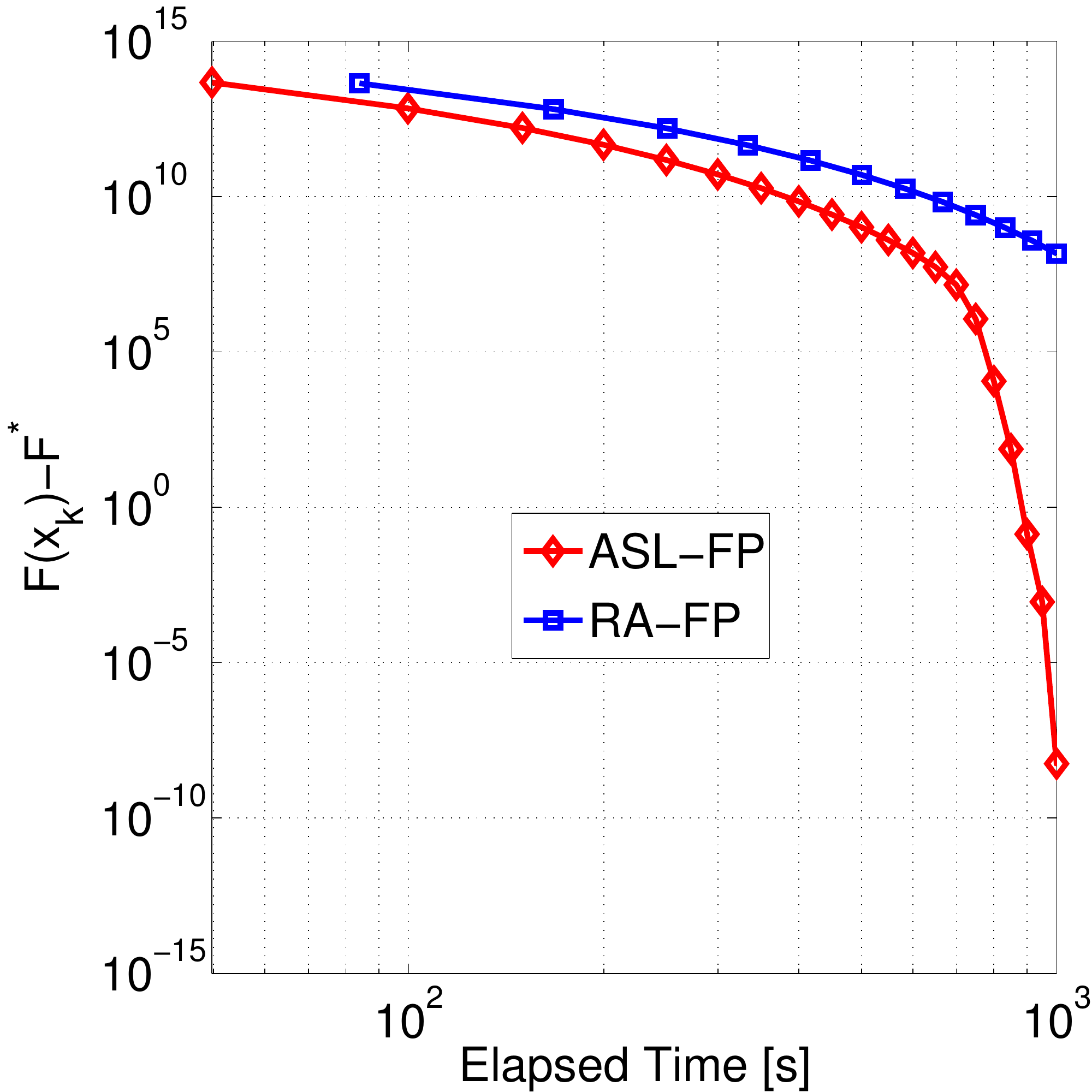}
 \caption{Evolution of $F(x_k)-F^*$ in time. ASL-FP significantly  outperforms  RA-FP. The loss $F$ is pushed down by  25 degrees of magnitude in less than 30 minutes (3TB problem).}
 \label{fig:ASL_vs_RA}
\end{figure}
}

\newpage
\section{Conclusions}
Overall, distributed algorithms can be both very efficient and
easy to implement, when one picks the right approach.
The first steps taken by the present authors over the past two years
seem to have been validated by the considerable interest \cite{richtarik2013distributed, fercoq2014fast, jaggi2014communication}
they have generated.

\bibliographystyle{ieeetr}
\bibliography{literature}

\newpage
\appendix

\section*{Notation Glossary} 

\begin{center}

\begin{tabular}{cp{10cm}c}
\multicolumn{1}{c}{$\qquad\qquad\qquad$}&  \multicolumn{1}{c}{}&  \multicolumn{1}{c}{$\qquad\qquad\qquad$}\\
& {\bf Optimization problem}  & \\[2mm]
$N$ & dimension of the optimization variable & \eqref{eq:P}\\
 $x, h$ &  vectors in $\R^N$ & \\
$F$ & $F=f+\cPsi$ (loss / objective function)  & \eqref{eq:P}\\
$F^*$  & optimal value, we assume $F^*> -\infty$ & \\
$f$ & smooth convex function ($f: \R^N \to \R$) & \eqref{eq:P}\\
$\cPsi$ & convex block separable function ($\cPsi: \R^N \to \R\cup \{+\infty\}$) & \eqref{eq:P}\\[2mm]

& {\bf Block structure} &\\[2mm]
$n$ & number of blocks & \\
$[n]$ & $[n]=\{1,2,\dots,n\}$ (the set of blocks) & Sec~\ref{sec:block_structure}\\
$N_i$ & dimension of block $i$ ($N_1+\dots+N_n = N$) & Sec~\ref{sec:block_structure}\\
$U_i$ & an $N_i \times N$ column submatrix of the $N \times N$ identity matrix& Sec~\ref{sec:block_structure}\\
$x^{(i)} $ & $x^{(i)}=U_i^T x \in\R^{N_i}$ (block $i$ of vector $x$)&Sec~\ref{sec:block_structure}\\
$\nabla_i f(x)$ & $\nabla_i f(x) = U_i^T \nabla f(x)$ (block gradient of $f$ associated with block $i$)& Sec~\ref{sec:block_structure}\\

$L_i$ & block Lipschitz constant of the gradient of $f$ & \eqref{eq:f_iLipschitzder}\\
$L$ & $L = (L_1,\dots,L_n)^T \in \R^n$ (vector of block Lipschitz constants)&\\
$w$ & $w = (w_1,\dots,w_n)^T \in \R^n$ (vector of positive weights) &\\

$\|x\|_w$ &  $\|x\|_w=(\sum_{i=1}^n w_i \|x^{(i)}\|^2_{(i)})^{1/2}$ (weighted norm associated with $x$)& \eqref{eq:norms}\\
$\cPsi_i$ & $i$-th componet of $\cPsi = \Psi_1 + \dots + \cPsi_n$ & \eqref
{eq:Psi_block_def}\\
$\mu_{\cPsi}(W)$ & strong convexity constant of $\cPsi$ with respect to the norm $\|\cdot\|_w$ & \eqref{eq:strong_def}\\
$\mu_f(W)$ & strong convexity constant of $f$ with respect to the norm $\|\cdot\|_w$ & \eqref{eq:strong_def}\\

$ J$ & subset  of $\{1,2,\dots,n\}$ & \\
$x_{[Z]}$ & vector in $\R^N$ formed from $x$ by zeroing out blocks $x^{(i)}$ for $i \notin Z$ & \eqref{eq:lllop09} \\[2mm]
 
& {\bf Block samplings} & \\[2mm]

$\omega$ & degree of partial separability of
 $f$ & Assumption  \ref{ass:f}\\

$\hat{Z}, Z_k$ & distributed block samplings (random subsets of $\{1,2,\dots,n\}$) &  Sec \ref{sec:algorithm}\\

$\K$ & number of nodes (partitions) &  Sec \ref{sec:algorithm}
\\
$\tau$ & \# of blocks updated in 1 iteration within one partition& \\

$\{\vc{\Partc}{\Kidx}\}_{\Kidx=1}^\K$ & partition of $[n]$ onto $\K$ parts  &  \\[2mm]
  
& {\bf Algorithm} & \\[2mm]

$\beta$ &  stepsize parameter depending on $f$ and $\hat{Z}$   & \\
  
$h^{(i)}(x)$ & $h^{(i)}(x) = (h(x))^{(i)} = \arg \min_{t \in \R^{N_i}} \ve{\nabla_i f(x)}{t} + \tfrac{\beta w_i}{2}\|t\|_{(i)}^2 + \cPsi_i(x^{(i)}+t)$ & \eqref{eq:h_definitionSeparablePart}\\
 
\end{tabular}
\end{center}

\end{document}